\documentclass[a4,12pt]{amsart}
\usepackage{amssymb,amstext,amsmath,amscd,amsthm,amsfonts,enumerate,graphicx,latexsym}
\usepackage[usenames]{color}
\newtheorem{thm}{Theorem}[section]
\newtheorem{cor}[thm]{Corollary}
\newtheorem{lem}[thm]{Lemma}
\newtheorem{prop}[thm]{Proposition}
\theoremstyle{definition}
\newtheorem{dfn}[thm]{Definition}
\newtheorem{rem}[thm]{Remark}

\newtheorem*{claim*}{Claim}
\theoremstyle{remark}
\newtheorem*{ac}{Acknowlegments}
\numberwithin{equation}{thm}
\def\height{\operatorname{\mathsf{ht}}}
\def\p{\mathfrak{p}}
\def\nf{\operatorname{\mathsf{NF}}}
\def\ipd{\operatorname{\mathsf{IPD}}}
\def\spec{\operatorname{\mathsf{Spec}}}
\def\supp{\operatorname{\mathsf{Supp}}}
\def\sing{\operatorname{\mathsf{Sing}}}
\def\R{\mathsf{R}}
\def\Ext{\operatorname{\mathsf{Ext}}}
\def\v{\operatorname{\mathsf{V}}}
\def\xx{\text{\boldmath{$x$}}}
\def\syz{\mathsf{\Omega}}
\def\codim{\operatorname{\mathsf{codim}}}
\def\tr{\mathsf{Tr}}
\def\depth{\operatorname{\mathsf{depth}}}
\def\dim{\operatorname{\mathsf{dim}}}
\def\pd{\operatorname{\mathsf{pd}}}
\def\lcm{\operatorname{\mathsf{\underline{MCM}}}}
\def\Hom{\operatorname{\mathsf{Hom}}}
\begin{document}
\allowdisplaybreaks
\setlength{\baselineskip}{15pt}
\title[MCM approximations and Serre's condition]{Maximal Cohen--Macaulay approximations\\
and Serre's condition}
\date{September 4, 2014}
\author{Hiroki Matsui}
\address{Graduate School of Mathematics, Nagoya University, Furocho, Chikusaku, Nagoya, Aichi 464-8602, Japan}
\email{m14037f@math.nagoya-u.ac.jp}
\author{Ryo Takahashi}
\address{Graduate School of Mathematics, Nagoya University, Furocho, Chikusaku, Nagoya, Aichi 464-8602, Japan}
\email{takahashi@math.nagoya-u.ac.jp}
\urladdr{http://www.math.nagoya-u.ac.jp/~takahashi/}
\thanks{RT was partly supported by JSPS Grant-in-Aid for Scientific Research (C) 25400038}
\dedicatory{Dedicated to Professor Ngo Viet Trung on the occasion of his sixtieth birthday}
\begin{abstract}
This paper studies the relationship between Serre's condition $(\R_n)$ and Auslander--Buchweitz's maximal Cohen--Macaulay approximations.
It is proved that a Gorenstein local ring satisfies $(\R_n)$ if and only if every maximal Cohen--Macaulay module is a direct summand of a maximal Cohen--Macaulay approximation of a (Cohen--Macaulay) module of codimension $n+1$.
\end{abstract}
\maketitle
%\tableofcontents
%%%%%%%%%%%%%%%%%%%%%%%%%%%%%%%%%%%%%%%%%%%%%%%%%%%%%%%%
\section{Introduction}

In the 1980s, Auslander and Buchweitz \cite{AB} introduced the notion of a maximal Cohen--Macaulay approximation of a finitely generated module over a Cohen--Macaulay local ring with a canonical module, which has been playing a fundamental role in the representation theory of Cohen--Macaulay rings.
Several years ago Kato \cite{K} gave the following characterization theorem of Gorenstein local rings by maximal Cohen--Macaulay approximations.
We abbreviate Cohen--Macaulay to CM and maximal Cohen--Macaulay to MCM.

\begin{thm}[Kato]
Let $R$ be a $d$-dimensional Gorenstein local ring.
\begin{enumerate}[\rm(1)]
\item
The following are equivalent for $d\ge1$.
\begin{enumerate}[\rm(a)]
\item
$R$ is a domain.
\item
Every MCM $R$-module is a MCM approximation of a (CM) $R$-module of codimension $1$.
\end{enumerate}
\item
The following are equivalent for $d\ge2$.
\begin{enumerate}[\rm(a)]
\item
$R$ is a unique factorization domain.
\item
Every MCM $R$-module is a MCM approximation of a (CM) $R$-module of codimension $2$.
\end{enumerate}
\end{enumerate}
\end{thm}

It is natural to ask what happens if in the statements (b) of the above theorem 
we weaken the condition of being a MCM approximation to that of being a direct summand of a MCM approximation.
The main purpose of this paper is to answer this question in more general settings.
Our main results yield the following theorem.

\begin{thm}\label{4140048}
Let $R$ be a $d$-dimensional Gorenstein local ring.
The following are equivalent for each $0\le c\le d$.
\begin{enumerate}[\rm(1)]
\item
$R$ satisfies Serre's condition $(\R_{c-1})$.
\item
Every MCM $R$-module is a direct summand of a $c$-th syzygy of a (CM) $R$-module of codimension $c$.
\item
Every MCM $R$-module is a direct summand of a MCM approximation of a (CM) $R$-module of codimension $c$.
\end{enumerate}
\end{thm}

Letting $c=1,2$ in the above theorem, we obtain the following result which is analogous to Kato's theorem.
This gives the answer to the question raised above.

\begin{cor}
Let $R$ be a $d$-dimensional Gorenstein local ring.
\begin{enumerate}[\rm(1)]
\item
The following are equivalent for $d\ge1$.
\begin{enumerate}[\rm(a)]
\item
$R$ is reduced.
\item
Every MCM $R$-module is a direct summand of a MCM approximation of a (CM) $R$-module of codimension $1$.
\end{enumerate}
\item
The following are equivalent for $d\ge2$.
\begin{enumerate}[\rm(a)]
\item
$R$ is normal.
\item
Every MCM $R$-module is a direct summand of a MCM approximation of a (CM) $R$-module of codimension $2$.
\end{enumerate}
\end{enumerate}
\end{cor}

This paper is organized as follows.
In Section 2, we consider over a CM local ring the condition that all MCM modules are direct summands of syzygies of certain modules.
In Section 3, we study over a Gorenstein local ring the condition that all MCM modules are direct summands of MCM approximations of certain modules.
The proof of Theorem \ref{4140048} is given at the end of this section.

%%%%%%%%%%%%%%%%%%%%%%%%%%%%%%%%%%%%%%%%%%%%%%%%%%%%%%%%
\section{MCM modules that are direct summands of syzygies}

Throughout this paper, let $R$ be a commutative Cohen--Macaulay local ring of Krull dimension $d$.
All $R$-modules are assumed to be finitely generated.

Let us begin with recalling some basic definitions.

\begin{dfn}
(1) For an integer $n\ge0$ we denote by $\syz^nM$ an {\em $n$-th syzygy} of $M$, that is, the image of the $n$-th differential map in a free resolution of $M$.\\
(2) For an integer $n\ge-1$ we say that $R$ satisfies {\em Serre's condition $(\R_n)$} if the local ring $R_\p$ is regular for all prime ideals $\p$ of $R$ with $\height\p\le n$.\\
(3) The {\em singular locus} $\sing R$ of $R$ is by definition the set of prime ideals $\p$ of $R$ such that the local ring $R_\p$ is nonregular.\\
(4) Let $M$ be an $R$-module.
The {\em nonfree locus} $\nf(M)$ (respectively, the {\em infinite projective dimension locus} $\ipd(M)$) of $M$ is defined as the set of prime ideals $\p$ of $R$ such that the $R_\p$-module $M_\p$ is nonfree (respectively, is of infinite projective dimension).\\
(5) Let $V$ be a closed subset of $\spec R$.
Then we set $\codim V=d-\dim V$ and call this the {\em codimension} of $V$.
The {\em codimension} $\codim M$ of an $R$-module $M$ is defined as the codimension of $\supp M$, whence $\codim M=d-\dim M$.
\end{dfn}

\begin{rem}
(1) If $X,Y$ are $n$-th syzygies of an $R$-module $M$, then $X\oplus F\cong Y\oplus G$ for some free $R$-modules $F,G$.\\
(2) By definition $R$ always satisfies $(\R_{-1})$.\\
(3) It is well-known and easy to see that the nonfree locus and the infinite projective dimension locus of an $R$-module are always closed subsets of $\spec R$ in the Zariski topology.\\
(4) If $M$ is a MCM $R$-module, then $\nf(M)$ is contained in $\sing R$.
\end{rem}

In the following proposition we study how to represent each MCM module as a direct summand of a syzygy of a certain CM module.
This result will become a basis of our main results.

\begin{prop}\label{4131834}
Let $M$ be a MCM $R$-module.
Then for each integer $0\le c\le \codim\nf(M)$ there exists a CM $R$-module $N$ such that
\begin{enumerate}[\rm(1)]
\item
$\codim N=c$,
\item
$\ipd(N) = \nf(M) $ and 
\item
$M$ is isomorphic to a direct summand of a $c$-th syzygy of $N$.
\end{enumerate}
\end{prop}

\begin{proof}
By virtue of \cite[Remark 5.2(1)]{dim}, there exists an ideal $I$ of $R$ with $\nf(M)=\v(I)$ such that $I\cdot\Ext_R^i(M,X)=0$ for all integers $i>0$ and all $R$-modules $X$.
As
$$
\dim\nf(M)=\dim R/I=d-\height I,
$$
we have $\height I=\codim\nf(M)\ge c$, and can take an $R$-sequence $\xx=x_1,\dots,x_c$ in $I$.
Setting $N=M/\xx M$, we see from \cite[Proposition 2.2]{stcm} that $M$ is isomorphic to a direct summand of $\syz^cN$.
The condition (3) is thus satisfied, and it is observed that $N$ is a CM $R$-module with $\codim N=d-\dim N=c$.

Now it remains to verify that $N$ satisfies the condition (2).
Fix a prime ideal $\p$ in the union $\ipd(N)\cup\nf(M)$.
Then it is easily observed that $\p$ contains the sequence $\xx$.
Hence by \cite[Exercise 1.3.6]{BH} the equalities
$$
\pd_{R_{\p}} N_{\p}
=\pd_{R_{\p}} M_{\p}/\xx M_{\p}
=\pd_{R_{\p}} M_{\p}+c
$$
hold.
This shows that the $R_\p$-module $N_\p$ has infinite projective dimension if and only if so does $M_\p$.
Since $M$ is a MCM $R$-module, the Auslander--Buchsbaum formula implies $\ipd(M)=\nf(M)$.
Therefore we obtain $\ipd(N)=\nf(M)$.
\end{proof}

As an immediate consequence of the above proposition, the following holds.

\begin{cor}\label{2208}
Let $M$ be a MCM $R$-module whose nonfree locus has dimension $n$.
Then there exists a CM $R$-module $N$ of dimension $n$ such that $M$ is isomorphic to a direct summand of $\syz^{d-n}N$.
\end{cor}

\begin{proof}
We have $\codim\nf(M)=d-n$.
Apply Proposition \ref{4131834} to $c:=d-n$.
\end{proof}

Applying the above corollary to $n=0$, we obtain the following result, which recovers \cite[Corollary 2.6]{stcm}.

\begin{cor}
Let $M$ be a MCM $R$-module which is locally free on the punctured spectrum of $R$.
Then there exists an $R$-module $N$ of finite length such that $M$ is isomorphic to a direct summand of $\syz^dN$.
\end{cor}

Next we establish a criterion for $R$ to satisfy Serre's condition $(\R_n)$ in terms of the codimensions of the nonfree loci of MCM $R$-modules.

\begin{prop}\label{4131835}
The following are equivalent for each $0\le c\le d$.
\begin{enumerate}[\rm(1)]
\item 
The ring $R$ satisfies $(\R_{c-1})$.
\item
One has $\codim\sing R\ge c$.
\item
One has $\codim\nf(M)\ge c$ for all MCM $R$-modules $M$.
\end{enumerate}
\end{prop}

\begin{proof}
(1) $\Rightarrow$ (2): Let $\p$ be a prime ideal in $\sing R$.
As $R$ satisfies $(\R_{c-1})$, the height of $\p$ is at least $c$, whence $\dim R/\p\le d-c$.
Therefore $\dim\sing R\le d-c$, which means that $\sing R$ has codimension at least $c$.

(2) $\Rightarrow$ (3): Since $\nf(M)$ is contained in $\sing R$, we have $\dim\nf(M)\le\dim\sing R$.
Hence the (in)equalities
$$
\codim\nf(M)=d-\dim\nf(M)\ge d-\dim\sing R=\codim\sing R\ge c
$$
follow.

(3) $\Rightarrow$ (1): Let $\p$ be a prime ideal of $R$ with $\height\p\le c-1$.
Let $M$ be a $d$-th syzygy of the $R$-module $R/\p$.
Then $M$ is a MCM $R$-module, and by assumption we have $\codim\nf(M)\ge c$, or equivalently,
$$
\dim\nf(M)\le d-c.
$$
Suppose that $R_{\p}$ is not regular.
Then the $R_\p$-module $M_{\p}$ is not free, for it is a $d$-th syzygy of the $R_\p$-module $\kappa(\p)$.
Hence $\p$ belongs to $\nf(M)$, and there are inequalities
$$
\dim\nf(M)\ge \dim R/\p\ge d-c+1.
$$
This contradiction shows that $R_\p$ is regular.
\end{proof}

Let us now state and prove the main result of this section, which characterizes CM local rings satisfying Serre's $(\R_n)$-condition.

\begin{thm}\label{4131833}
For every integer $0\le c\le d$ the following are equivalent.
\begin{enumerate}[\rm(1)]
\item
The ring $R$ satisfies $(\R_{c-1})$.
\item
Every MCM $R$-module is isomorphic to a direct summand of a $c$-th syzygy of a CM $R$-module of codimension $c$.
%\item
%Every MCM $R$-module is isomorphic to a direct summand of a $c$-th syzygy of an $R$-module of codimension $c$.
\item
Every MCM $R$-module is isomorphic to a direct summand of some syzygy of an $R$-module of codimension at least $c$.
\end{enumerate}
\end{thm}

\begin{proof}
Propositions \ref{4131834} and \ref{4131835} show that (1) implies (2), and it is obvious that (2) implies (3).
Assume that (3) holds, and take any MCM $R$-module $M$.
By assumption, there are an $R$-module $N$ with $\codim N\ge c$ and an integer $b\ge0$ such that $M$ is isomorphic to a direct summand of $\syz^bN$.
Then we have inclusions $\nf(M)\subseteq\nf(\syz^bN)\subseteq\supp N$ of closed subsets of $\spec R$, which implies
$$
\dim\nf(M)\le \dim\nf(\syz^bN)\le \dim\supp N=\dim N\le d-c.
$$
Hence $\nf(M)$ has codimension at least $c$, and it is deduced from Proposition \ref{4131835} that $R$ satisfies $(\R_{c-1})$.
\end{proof}

%%%%%%%%%%%%%%%%%%%%%%%%%%%%%%%%%%%%%%%%%%%%%%%%%%%%%%%%
\section{MCM modules that are direct summands of MCM approximations}

Throughout this section, our ring $R$ is further assumed to be Gorenstein.
The following is a celebrated result due to Auslander and Buchweitz \cite[Theorem 1.8]{AB}.

\begin{thm}[Auslander--Buchweitz]
For each $R$-module $M$ there exists an exact sequence
\begin{equation}\label{4131613}
0 \to Y \to X \to M \to 0
\end{equation}
of $R$-modules such that $X$ is MCM and $Y$ has finite projective dimension.
\end{thm}

\begin{dfn}
A MCM $R$-module $X$ admitting an exact sequence of the form \eqref{4131613} is called a {\em MCM approximation} of $M$.
\end{dfn}

For an $R$-module $M$ we denote by $\tr M$ the {\em (Auslander) transpose} of $M$, that is, the cokernel of the $R$-dual of the first differential map in a free resolution of $M$.
We denote by $\lcm(R)$ the {\em stable category of MCM $R$-modules}.
This is defined as follows: the objects of $\lcm(R)$ are precisely the MCM $R$-modules, and the hom-set $\Hom_{\lcm(R)}(M,N)$ of objects $M,N$ in $\lcm(R)$ is the quotient of $\Hom_R(M,N)$ by the $R$-submodule consisting of homomorphisms factoring through free $R$-modules.
Since $R$ is assumed to be Gorenstein, $\lcm(R)$ is a triangulated category, and taking a syzygy and a transpose defines an autoequivalence and a duality of $\lcm(R)$, respectively.
\begin{align*}
\syz:&\lcm(R)\xrightarrow{\cong}\lcm(R),\\
\tr:&\lcm(R)\xrightarrow{\cong}\lcm(R)^{\mathsf{op}}.
\end{align*}
For details, we refer the reader to \cite{ABr} and \cite{B}.

One can describe a MCM approximation by using syzygies and transposes:

\begin{lem}\label{4131913}
For any $R$-module $M$, the $R$-module
$$
\tr\syz^n\tr\syz^nM
$$
is a MCM approximation of $M$ for all $n\ge d-\depth M$.
\end{lem}

\begin{proof}
Note that $\syz^nM$ is a MCM $R$-module.
Since both $\syz$ and $\tr$ preserve the MCM property, the $R$-module $X=\tr\syz^n\tr(\syz^nM)$ is also a MCM module.
It follows from \cite[Proposition (2.21) and Corollary (4.22)]{ABr} that there exists an exact sequence
\begin{equation}\label{4131855}
0 \to K \to X \to M \to 0
\end{equation}
of $R$-modules such that $K$ has projective dimension at most $n-1$.
Consequently, $X$ is a MCM approximation of $M$.
\end{proof}

A MCM approximation version of Corollary \ref{2208} also holds true:

\begin{prop}
\begin{enumerate}[\rm(1)]
\item
Let $M$ be a MCM $R$-module with $n$-dimensional nonfree locus.
Then there exists an $n$-dimensional CM $R$-module $N$ such that $M$ is isomorphic to a direct summand of a MCM approximation of $N$.
\item
Let $M$ be a MCM $R$-module that is locally free on the punctured spectrum of $R$.
Then there exists an $R$-module $N$ of finite length such that $M$ is isomorphic to a direct summand of a MCM approximation of $N$.
\end{enumerate}
\end{prop}

\begin{proof}
(1) It is easy to see that $\nf(\syz^{d-n}M)$ coincides with $\nf(M)$.
Applying Corollary \ref{2208} to the MCM module $\syz^{d-n}M$, we find a CM module $N$ of dimension $n$ such that $\syz^{d-n}M$ is isomorphic to a direct summand of $\syz^{d-n}N$.
Taking $\tr\syz^{d-n}\tr$ yields that $M$ is isomorphic to a direct summand of
$$
X:=\tr\syz^{d-n}\tr\syz^{d-n}N\oplus F
$$
for some free $R$-module $F$.
Using Lemma \ref{4131913}, we easily see that $X$ is a MCM approximation of $N$.

(2) The assertion follows from applying (1) to $n=0$.
\end{proof}

The main result of this section is the following characterization of Gorenstein local rings satisfying Serre's condition $(\R_n)$.
This result can be viewed as a MCM approximation version of Theorem \ref{4131833}.

\begin{thm}\label{4140050}
The following are equivalent for each $0\le c\le d$.
\begin{enumerate}[\rm(1)]
\item
$R$ satisfies $(\R_{c-1})$.
\item
Every MCM $R$-module is isomorphic to a direct summand of a MCM approximation of a CM $R$-module of codimension $c$.
\item
Every MCM $R$-module is isomorphic to a direct summand of a MCM approximation of an $R$-module of codimension at least $c$.
\end{enumerate}
\end{thm}

\begin{proof}
(1) $\Rightarrow$ (2): Let $M$ be a MCM $R$-module.
Using Theorem \ref{4131833} for the MCM $R$-module $\syz^cM$, we get a CM $R$-module $N$ of codimension $c$ such that $\syz^cM$ is isomorphic to a direct summand of $\syz^cN$.
Then applying $\tr\syz^c\tr$ to this relation shows that $\tr\syz^c\tr\syz^cM$ is isomorphic to a direct summand of $X:=\tr\syz^c\tr\syz^cN$ up to free summands.
By Lemma \ref{4131913} the module $X$ is a MCM approximation of $N$.
Since we have a duality
$$
\tr\syz^c:\lcm(R)\xrightarrow{\cong}\lcm(R),
$$
the $R$-module $\tr\syz^c\tr\syz^cM$ is isomorphic to $M$ up to free summands.
Therefore $M$ is isomorphic to a direct summand of $X\oplus F$ for some free $R$-module $F$.
It is easy to see that $X\oplus F$ is also a MCM approximation of $N$.

(2) $\Rightarrow$ (3): This implication is obvious.

(3) $\Rightarrow$ (1): Let $M$ be a MCM $R$-module.
Then $N:=\tr\syz^d\tr M$ is also a MCM $R$-module.
Applying the condition (3) to $N$, we observe that there exists an $R$-module $L$ of codimension at least $c$ such that $N$ is isomorphic to a direct summand of a MCM approximation $X$ of $L$.
It follows from \cite[Theorem B]{AB} and Lemma \ref{4131913} that the $R$-module $X$ is isomorphic to $\tr\syz^d\tr\syz^dL$ up to free summands.
The functor
$$
\tr\syz^d\tr:\lcm(R)\xrightarrow{\cong}\lcm(R)
$$
is an equivalence, so we see that $M$ is isomorphic to a direct summand of $\syz^dL$ up to free summands.
Thus Theorem \ref{4131833} implies that $R$ satisfies Serre's condition $(\R_{c-1})$.
\end{proof}

\begin{proof}[Proof of Theorem \ref{4140048}]
The assertion follows by combining Theorems \ref{4131833} and \ref{4140050}.
\end{proof}

\begin{ac}
The authors are grateful to Olgur Celikbas for his helpful comments.
The authors also thank the referee for his/her careful reading.
\end{ac}

%%%%%%%%%%%%%%%%%%%%%%%%%%%%%%%%%%%%%%%%%%%%%%%%%%%%%%%%

%%%%%%%%%%%%%%%%%%%%%%%%%%%%%%%%%%%%%%%%%%%%%%%%%%%%%%%%
\end{document}